\documentclass[12pt, a4paper, english, reqno]{amsart}
\usepackage{amssymb} 
\usepackage[dvipsnames]{xcolor}
\usepackage{mathtools}
\usepackage[linktocpage]{hyperref} 
\usepackage{tikz-cd}
\usetikzlibrary{decorations.pathreplacing}

\usepackage[shortlabels]{enumitem}
\usepackage[cal=cm,scr=euler]{mathalfa}

\usepackage[margin=1.20in]{geometry}
\usepackage{microtype}
\usepackage{mlmodern}

\hypersetup{
    colorlinks = true,
    linkbordercolor = {white},
    linkcolor = {NavyBlue},
    anchorcolor = {black},
    citecolor = {NavyBlue},
    filecolor = {cyan},
    menucolor = {red},
    urlcolor = {NavyBlue}
}

\newcommand\QQ{{\mathbb Q}}
\newcommand{\des}{\operatorname{des}}

\newcommand\R{{\mathbb R}}
\newcommand\Z{{\mathbb Z}}
\renewcommand\P{{\mathcal{P}}}
\renewcommand\O{{\mathcal{O}}}
\newcommand{\Ehr}{\operatorname{Ehr}}
\newcommand{\comb}{\operatorname{comb}}
\newcommand{\Pyr}{\operatorname{Pyr}}
\newcommand{\Ch}{\operatorname{Ch}}

\newcommand{\myeuler}[2]{
  \mathchoice
    {\hbox{$\bigg\langle\mkern-9mu\bigg\langle$} \genfrac{}{}{0pt}{0}{#1}{#2} \hbox{$\bigg\rangle\mkern-9mu\bigg\rangle$}}
    {\hbox{$\Big\langle\mkern-7mu\Big\langle$} \genfrac{}{}{0pt}{1}{#1}{#2} \hbox{$\Big\rangle\mkern-7mu\Big\rangle$}}
    {\langle\langle \genfrac{}{}{0pt}{2}{#1}{#2} \rangle\rangle}
    {\langle\langle \genfrac{}{}{0pt}{3}{#1}{#2} \rangle\rangle}
}

\theoremstyle{plain}
\newtheorem{teo}{Theorem}[section]
\newtheorem{theorem}[teo]{Theorem}
\newtheorem{corollary}[teo]{Corollary}
\newtheorem{lemma}[teo]{Lemma}

\newtheorem{proposition}[teo]{Proposition}

\theoremstyle{definition}
\newtheorem{definition}[teo]{Definition}
\newtheorem{example}[teo]{Example}
\newtheorem{remark}[teo]{Remark}

\numberwithin{equation}{section}

\title[An identity for second Eulerian numbers]{An identity for second Eulerian numbers\\via lattice-point counting}

\author{Jack Boncompagni}
\address{Dipartimento di Matematica, Universit\`a di Pisa, Pisa, Italy
}
\email{j.boncompagni@studenti.unipi.it}

\allowdisplaybreaks

\begin{document}

\begin{abstract}
    The second Eulerian numbers are defined via the descent enumerator of Stirling permutations, a class of permutations introduced by Gessel and Stanley. We give a simple and conceptual proof of two identities relating the Bernoulli numbers and the second Eulerian numbers. We rely on a recent Ehrhart-theoretic idea of Ferroni.
\end{abstract}

\maketitle

\section{Introduction}\label{sec:intro}

\noindent The main purpose of this article is to give a simple and conceptual proof of the following result:

\begin{theorem}\label{thm:1.1}
    For every positive integer $n\ge 1$, the following identities hold
    \[2B_{n}=\sum_{j=0}^{n-1}\frac{(-1)^j}{\binom{2n-1}{j}} 
    \myeuler{n}{j},\]
    \[\sum_{j=1}^n\frac{B_j}{j}=
    H_{2n} - n +\frac{1}{2n} \sum_{j=1}^{n-1} \frac{(-1)^{j-1}}{\binom{2n-1}{j-1}}\myeuler{n}{j}.
    \]
\end{theorem}

\noindent \looseness=-1 The numbers $B_n$ appearing on the left-hand side are the Bernoulli numbers, whereas $H_n$ are the harmonic numbers and $\myeuler{n}{j}$ denotes a second Eulerian number, i.e., the number of Stirling permutations on $2n$ elements (as defined in \cite{GesselStanley}) having exactly $j$ descents. Fu \cite{Fu2021} proved the first identity, although the proof does not explain how one might come up with other identities of the same type; see also \cite{MO45756}.

In a recent paper by Ferroni \cite{Ferroni2026}, a general technique for producing identities such as those in Theorem~\ref{thm:1.1} is provided. The crucial observation is that several identities of similar form can be deduced from studying the Ehrhart polynomial of a suitable polytope. Ferroni also raised a broad challenge of finding further applicability of these ideas, and the present article illustrates how to use them to prove (and, more importantly, explain) Theorem~\ref{thm:1.1}.

To demonstrate the versatility of this approach, later in the article we introduce a new class of polytopes to which the technique used for Theorem~\ref{thm:1.1} naturally extends.

\section{Preliminaries}

\subsection{Bernoulli numbers}

First, we briefly recall the basic facts about the Bernoulli numbers (see OEIS \cite{oeis}: \href{https://oeis.org/A027641}{A027641} and \href{https://oeis.org/A027642}{A027642}).\footnote{Note that we adopt the convention where $B_1 = 1/2$. Some authors prefer the alternative convention, which differs only in the sign of $B_1$.}

\begin{definition}
    The sequence of \textbf{Bernoulli numbers} $\{B_n\}_{n \ge 0}$ is defined recursively as follows:
    \[
    B_0 = 1 \quad \text{and} \quad B_n = 1 - \frac{1}{n+1} \sum_{j=0}^{n-1} \binom{n+1}{j} B_j \quad \text{for } n \ge 1.
    \]
\end{definition}

\noindent These numbers satisfy numerous properties (see \cite[Section 6.5]{GKP94}), one of which is of central importance to this work. To state it, we first introduce a special class of polynomials known as Faulhaber polynomials.

\begin{definition}
For each integer $n \ge 1$, the \textbf{Faulhaber polynomial} 
$F_{n}(x)\in\QQ[x]$ is the unique polynomial of degree $n+1$ such that
\[
F_{n}(k)=1^{n}+2^{n}+\cdots+k^{n}
\]
for all non-negative integers $k$. For example, $F_{1}(x)=\frac{x(x+1)}{2}$, 
$F_{2}(x)=\frac{x(x+1)(2x+1)}{6}$.
\end{definition}

\noindent
Using the standard notation $[x^j] f(x)$ for the coefficient of $x^j$ in a polynomial or formal power series $f(x)$, the fundamental property is that, for each $n\ge 1$,
\begin{equation}\label{Bernoulli}
    B_{n}= [x]F_{n}(x). 
\end{equation}

\subsection{Second Eulerian numbers}
The second Eulerian numbers (see OEIS \cite{oeis}: \href{https://oeis.org/A340556}{A340556}) count descents in Stirling permutations, more precisely as follows.\footnote{Some authors conventionally treat the final index $i=2n$ as a descent, leading to an off-by-one shift in the notation.}

\begin{definition}[{Gessel--Stanley \cite{GesselStanley}}]
    A \textbf{Stirling permutation} of order $n$ is a permutation of the multiset $\{1, 1, 2, 2, \ldots, n, n\}$ such that for every $i \in \{1, \ldots, n\}$, all elements $s$ appearing between the two occurrences of $i$ satisfy $s > i$. We denote the set of all such permutations by $\mathcal{S}_n^{(2)}$.
\end{definition} 

\begin{definition}
    Let $\sigma \in \mathcal{S}_n^{(2)}$ be a Stirling permutation, written in one-line notation as~$\sigma_1\sigma_2\cdots\sigma_{2n}$. An index $i \in \{1, 2, \ldots, 2n-1\}$ is called a \textbf{descent} of $\sigma$ if $\sigma_i > \sigma_{i+1}$.
    We denote by $\des(\sigma)$ the number of descents of $\sigma$.
\end{definition}

\noindent For example, $\sigma=5513324421 \in \mathcal{S}_5^{(2)}$ has $\des(\sigma)=4$ descents, located at positions $2, 5, 8, \text{ and } 9$.

\begin{definition}
    The \textbf{second Eulerian numbers}, denoted by $\myeuler{n}{k}$, count the number of Stirling permutations of order $n$ with exactly $k$ descents.
\[\myeuler{n}{k}=  \#\{ \sigma \in \mathcal{S}_n^{(2)} \mid \des(\sigma)=k \} .\]
\end{definition}

\begin{remark}
    The second Eulerian numbers are a natural extension of the classical Eulerian numbers, which count the number of permutations of $\{1, \ldots, n\}$ with exactly $k$ descents.
\end{remark}

\noindent We refer to \cite[Section 6.2]{GKP94} for a detailed discussion of these numbers.

\subsection{Stirling numbers of the second kind}
We will see that these numbers are also closely related to Stirling numbers of the second kind (\cite[Section 6.1]{GKP94}), which can be expressed as specializations of a specific symmetric polynomial.
\begin{definition}
    The \textbf{Stirling numbers of the second kind}, denoted by $\left\{ \begin{smallmatrix} n \\ k \end{smallmatrix} \right\}$, count the number of ways to partition a set of $n$ elements into $k$ non-empty subsets.
\end{definition}

\begin{definition}\label{def:2.8}
The \textbf{complete homogeneous symmetric polynomial} of degree $n$, denoted by $h_n(x_1, \ldots , x_k)$, is defined by
\[h_n(x_1, x_2 , \ldots , x_k) \coloneqq \sum_{1 \le i_1 \le i_2 \le \cdots \le i_n \le k}
x_{i_1} x_{i_2} \cdots x_{i_n}\]
\[h_0(x_1, x_2 , \ldots , x_k) \equiv 1.\]
\end{definition}

\noindent Its generating function will also be useful for later propositions.
\begin{lemma}\label{lem:2.8}
    The generating function for the complete homogeneous symmetric polynomials $h_n(x_1, \ldots,x_k)$ is given by
    \[\sum_{j \ge 0} h_j(x_1, \ldots,x_k) t^j  =\prod_{i=1}^k\frac{1}{1-x_it}.\]
\end{lemma} 

\noindent This is a standard result, obtained by expanding the product $\prod (1-x_it)^{-1}$ into a formal power series.

These two objects are related by the following lemma.

\begin{lemma}\label{lem:2.10}
For all non-negative integers $n, k$ with $k \ge 1$, the Stirling numbers of the second kind satisfy
\[
\left\{ \begin{matrix} n+k \\ k \end{matrix} \right\} = h_n(1, 2, \ldots, k).
\]
\end{lemma}

\noindent The identity follows immediately by comparing the respective generating functions (see, for example, \cite[1.94c]{StanleyEC1}).

Before proceeding, we recall the definition of the harmonic numbers (\cite[Section 6.3]{GKP94}), which will appear in the following results.

\begin{definition}
For every integer $n \ge 1$, the $n$-th \textbf{harmonic number} $H_n$ is defined as the sum of the reciprocals of the first $n$ natural numbers:
\[H_n = \sum_{j=1}^n \frac{1}{j}.\]
\end{definition}

\section{Ehrhart theory and order polytopes}

\subsection{Ehrhart theory}

A \textbf{lattice polytope} $\P \subset \R^d$ is the convex hull of finitely many lattice points, i.e., points with integer coordinates. A common problem in discrete geometry is to enumerate the lattice points contained in a given lattice polytope.
\begin{definition}
    Let $\P \subset \R^n$ be a lattice polytope. The \textbf{dimension} of $\P$ is the dimension of its affine span. The \textbf{Ehrhart polynomial} of $\P$ is the function $\Ehr_\P(k)$ that counts the number of lattice points in the polytope $\P$ scaled by a factor $k \in \Z_{> 0}$:
    \[\Ehr_\P(k) \coloneqq \# (k\P \cap \Z^n).\]
\end{definition}

\noindent The following result is due to Eugene Ehrhart, in whose honor the Ehrhart polynomial is named.

\begin{theorem}
    Let $\P \subset \R^n$ be a lattice polytope of dimension $d$. The map $k\mapsto \Ehr_\P(k)$, defined for positive integers $k$, is a polynomial of degree $d$.
\end{theorem}

\noindent For a deeper exploration of this remarkable result, we refer the reader to \cite[Theorem 3.8]{ccd}.\footnote{
    Note that a different notation is adopted in the cited work; specifically, the Ehrhart polynomial is denoted by $L_\P(t)$, while $\Ehr_\P(z)$ denotes its generating function.}

We now introduce another polynomial associated with a given lattice polytope by considering its Ehrhart generating function. For every polynomial $p(x)$ of degree $d$, it is a standard result in enumerative combinatorics that its generating function is a rational function of the form: 
\[\sum_{j \ge 0} p(j)x^j = \frac{h(x)}{(1-x)^{d+1}},\]
where $h(x)$ is a polynomial of degree at most $d$ (see \cite[Corollary 4.3.1]{StanleyEC1}).
In the specific case $p(x)=\Ehr_\P(x)$, the numerator $h(x)$ is called the \textbf{\boldmath $h^*$-polynomial} of $\P$ and we set $h^*_i \coloneqq [x^i]h^*(x)$. Expanding the definition leads to the identity below.
\begin{lemma}\label{lem:3.3}
{\cite[Lemma 3.14]{ccd}}
    Let $\P \subset \R^n$ be a lattice polytope of dimension $d$. The Ehrhart polynomial of $\mathcal{P}$ is given by
    \[\Ehr_\P(x)=\sum_{j=0}^{d}h_{j}^{*}\binom{x+d-j}{d}.\]
\end{lemma}

\noindent We also require the following lemma,\footnote{
This is equivalent to the condition $\Ehr_\P(0) = 1$, which may be more or less immediate depending on the definitions used.
} whose proof can be found in \cite[Lemma 3.13]{ccd}.

\begin{lemma} \label{hzero}
    Let $\P \subset \R^n$ be a lattice polytope. Then the $h^*$-polynomial satisfies $h^*_0=1$.
\end{lemma}

\noindent With these results in hand, we are ready to establish a more specific connection between the Ehrhart polynomial and the $h^*$-polynomial.
This relationship, which was the central motivation of \cite{Ferroni2026}, serves as the core of the main proof.

\begin{lemma}\label{Lem_F}
 Let $\P \subset \R^n$ be a lattice polytope of dimension $d > 0$. Then
\[[x]\Ehr_\P(x-1) = (-1)^{d-1}h_d^* H_d + \frac{1}{d} \sum_{j=0}^{d-1} (-1)^j \frac{h_j^*}{\binom{d-1}{j}}.
\tag{1}\]
\[[x]\Ehr_\P(x) = H_d + \frac{1}{d} \sum_{j=1}^{d} (-1)^{j-1} \frac{h_j^*}{\binom{d-1}{j-1}}.\tag{2}\]
\end{lemma}

\begin{proof}
We begin with the first identity. By Lemma~\ref{lem:3.3}, we have
\begin{align*}
[x]\Ehr_\P(x-1) &= [x]\sum_{j=0}^d h_j^* \binom{x-1+d-j}{d} \\
&= [x]\frac{h_d^*}{d!}(x-1)(x-2)\cdots(x-d) \ + \\
&\ \quad[x]\sum_{j=0}^{d-1} \frac{h_j^*}{d!} (x+d-1-j)(x+d-2-j) \cdots (x-j).
\end{align*}

\noindent The coefficient of $x$ in the first term is given by
\[ (-1)^{d-1}\frac{h_d^* }{d!} \sum_{j=1}^d \frac{d!}{j} = (-1)^{d-1}h_d^* \sum_{j=1}^d \frac{1}{j} = (-1)^{d-1}h_d^* H_d.\]
 
\noindent For the second term, since the sum ranges from $j=0$ to $d-1$, the factor $x$ is present in every summand. Consequently, the coefficient of $x$ in this second part is given by
\[ \sum_{j=0}^{d-1} \frac{h_j^*}{d!} (d-1-j)! \cdot (-1)^j j!
= \frac{1}{d} \sum_{j=0}^{d-1} (-1)^j \frac{h_j^*}{\binom{d-1}{j}}, \]
which completes the proof of (1).

We argue similarly for the second identity:
\begin{align*}
[x]\Ehr_\P(x) &= [x]\sum_{j=0}^d h_j^* \binom{x+d-j}{d} \\
&= [x]\frac{h_0^*}{d!}(x+d)(x+d-1)\cdots(x+1) \ + \\
&\ \quad[x]\sum_{j=1}^d\frac{h_j^*}{d!} (x+d-j)(x+d-1-j) \cdots (x+1-j).
\end{align*}

\noindent The coefficient of $x$ in the first term is given by $h_0^*H_d$, which simplifies to $H_d$ by Lemma~\ref{hzero}. On the other hand, the coefficient of $x$ in the second term is the same as before, but with $j$ shifted by 1. This proves (2).
\end{proof}

\noindent A restricted version of the first identity is proved in \cite{Ferroni2026}, specifically for polytopes with no interior lattice points. In such cases, $h^*_d=0$, and therefore the first term on the right-hand side vanishes. As the reader will see, our proof ultimately relies on working with a polytope for which $h^*_d=0$, but we nevertheless provide the complete general statement for future reference.

\subsection{Order polytopes}

A remarkable application of Ehrhart theory arises in the study of finite posets. Given a \textbf{finite} poset $(P, <_P)$, the \textbf{order polynomial} $\Omega(P, k)$ counts the number of order-preserving maps from $P$ to $\{1, \ldots, k\}$. As shown by Stanley \cite{StanleyOP}, this combinatorial invariant can be realized geometrically through the \textbf{order polytope} $\O(P) \subset \R^{\lvert P \rvert}$, a full-dimensional polytope defined by the inequalities $0 \le x_i \le x_j \le 1$ for all $i <_P j$. The core connection lies in the identity:
\begin{equation} \label{eq_order}
    \Omega_P(k) = \Ehr_{\O(P)}(k-1).
\end{equation} 

\begin{remark}\label{rem_des}
    From this perspective, the $h^*$-polynomial of $\O(P)$ has a deep combinatorial interpretation: it is the descent enumerator of the linear extensions of $P$ (see \cite[Theorem 3.15.8]{StanleyEC1}).
\end{remark}

It is a classical result that for the order polytope $\O(P)$, associated with every \textbf{non-empty}\footnote{This is a boundary case, as the single point in $\R^0$ is open; the empty poset is used in Section~\ref{sec_4}.}
    finite poset $P$ of cardinality $d$, the Ehrhart polynomial satisfies
    \begin{equation} \label{eq_ehrmeno}
        \Ehr_{\O(P)}(-1)=0.
    \end{equation}
    This property follows (via the Ehrhart-Macdonald reciprocity \cite[Theorem~4.1]{ccd}) from the fact that the relative interior of $\O(P)$ contains no lattice points, since order polytopes have vertices with $0/1$-coordinates.
    Evaluating Lemma~\ref{lem:3.3} at $-1$ yields $h^*_d=0$; thus, the first identity of Lemma~\ref{Lem_F} simplifies, as the term involving the harmonic number vanishes.

\section{The main proof} \label{sec_4}

\noindent We now focus on the order polytope $P_n$, which we call the \textbf{comb polytope}, associated with the \textbf{comb poset} $C_n$, whose Hasse diagram is shown in Figure~\ref{fig:1}. By computing its Ehrhart polynomial and its $h^*$-polynomial, we conclude the proof of the main theorem via Lemma~\ref{Lem_F}.
\[P_n \coloneqq  \O(C_n)=
\{(x_1, \ldots , x_n, y_1, \ldots , y_n) \in [0,1]^{2n} \mid x_1 \leq \cdots \leq x_n , \  y_i \leq x_i \ \forall i\}.
\]

\noindent Although the poset $C_n$ (and its associated order polytope $P_n$) is well-defined for all integers $n \ge 1$, we adopt the convention $C_0 = \varnothing$. Under this convention, $P_0$ is the unique point in $\mathbb{R}^0$, and its Ehrhart polynomial is identically one: $\Ehr_{P_0}(x) \equiv 1$.

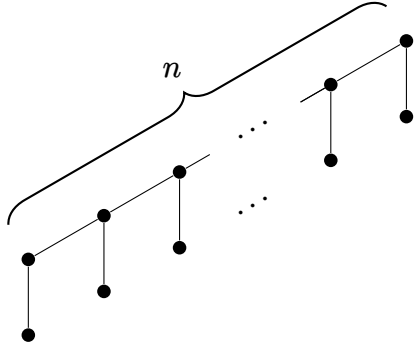
\begin{figure}[h]
    \centering
    \begin{tikzpicture}[
        every node/.style={circle, fill=black, inner sep=1.8pt},
        scale=1 
    ]
        \def\slope{0.577}

        \foreach \i in {0, 1, 2} {
            \node (L\i) at (\i, {\i*\slope - 1}) {};
            \node (U\i) at (\i, \i*\slope) {};   
            \draw (L\i) -- (U\i);
            \ifnum \i > 0
                \pgfmathtruncatemacro{\prev}{\i-1}
                \draw (U\prev) -- (U\i);
            \fi
        }

        \draw (U2) -- (2.4, {2.4*\slope});
        \node[draw=none, fill=none, rotate=30] at (3, {3*\slope}) {$\dots$}; 
        \draw (3.6, {3.6*\slope}) -- (4, {4*\slope});
        
        \node[draw=none, fill=none, rotate=30] at (3, {3*\slope - 1}) {$\dots$};

        \foreach \i in {4,5} {
            \node (L\i) at (\i, {\i*\slope - 1}) {};
            \node (U\i) at (\i, \i*\slope) {};
            \draw (L\i) -- (U\i);
            \ifnum \i > 4
                \pgfmathtruncatemacro{\prev}{\i-1}
                \draw (U\prev) -- (U\i);
            \fi
        }
        
        \draw (U4) -- (3.6, {3.6*\slope});

        \draw [
            thick,
            decorate, 
            decoration={brace, amplitude=10pt, raise=15pt}
        ] (0, 0) -- (5, {5*\slope});
        
        \node[draw=none, fill=none] at (1.9, 2.5) {$n$};
    \end{tikzpicture}
    \caption{Comb Poset $C_n$.}
    \label{fig:1}
\end{figure}

\subsection{The Ehrhart polynomial}
We want to compute the Ehrhart polynomial $\Ehr_{P_n}(k)$. By definition, for every non-negative integer $k$, this equals the number of lattice points in the $k$-th dilate of $P_n$. That is, 
\begin{align*}\Ehr_{P_n}(k) 
&= \# (kP_n \cap \ \mathbb{Z}^{2n})\\ 
&=
\# \{(x_1, \ldots , x_n, y_1, \ldots , y_n) \in [0,k]^{2n} \mid x_1 \leq \cdots \leq x_n, \ y_i \leq x_i \ \forall i\}.
\end{align*}

\begin{proposition}\label{prop:4.1}
For all non-negative integers $n, k$ with $k \ge 1$, the Ehrhart polynomial of the comb polytope satisfies
    \[
\Ehr_{P_n}(k-1)= h_n(1, 2, \ldots, k).
\]
\end{proposition}

\begin{proof}
    By convention, $\Ehr_{P_0}(k) = h_0(x_1, \ldots, x_k) = 1$; therefore, we assume $n \ge 1$.
    For a non-negative integer $k$, we want to count the lattice points $(x_1, \ldots , x_n, y_1, \ldots , y_n)$ such that $x_1 \leq \cdots \leq x_n$  and $0 \leq y_i \leq x_i \leq k$ for each $i=1 , \ldots , n$.
    We sum over all valid integer sequences $0 \leq x_1 \leq \cdots \leq x_n \leq k$, weighted by the number of possible choices for each $y_j$. Since each $y_j$ can independently take $x_j + 1$ values in $\{0, 1, \ldots, x_j\}$, the total count is
   \[\Ehr_{P_n}(k) = \sum_{0 \leq x_1 \leq \cdots \leq x_n \leq k} (x_1 + 1)(x_2 + 1) \cdots (x_n + 1).\]
    After the change of variables $z_i = x_i + 1$:
    \[\Ehr_{P_n}(k) = \sum_{1 \leq z_1 \leq \cdots \leq z_n \leq k+1} z_1 z_2 \cdots z_n = h_n(1, 2, \ldots, k+1),\]
    where the last equality follows from the specialization $x_i \mapsto i$ in Definition~\ref{def:2.8}.
\end{proof}

\begin{corollary}\label{cor:4.2}
    From Lemma~\ref{lem:2.10}, we have $\Ehr_{P_n}(k-1)= \left\{ \begin{matrix} n+k \\ k \end{matrix} \right\}$.
\end{corollary}

\noindent Although this is an elegant expression for the Ehrhart polynomial of $P_n$, its current form does not immediately reveal the coefficient $[k]\Ehr_{P_n}(k-1)$. However, to extract this coefficient we use standard generatingfunctionology, specifically the logarithmic series expansion.
\begin{remark}\label{rem:4.3}
    The power series expansion for the logarithm is given by \[ \ln(1+x) = \sum_{j \ge 1} (-1)^{j+1} \frac{x^j}{j}.\]
    For every formal power series $f(x)$ with constant term equal to one (i.e., $f(0) = 1$), the composition $\ln(f(x))$ is well-defined. Furthermore, we observe that the coefficient of $x$ remains invariant under this transformation: $[x]f(x) = [x]\ln(f(x))$.
\end{remark}

\begin{proposition}\label{Prop_coeff}
    For each integer $n \ge 1$, the following identity holds
    \[[x]\Ehr_{P_n}(x-1) = \frac{B_n}{n}.
\]
\end{proposition}

\begin{proof}
    To simplify the notation, we set $E_n(x) \coloneqq \Ehr_{P_n}(x-1)$.
    
    Let $E(x,t)$ be the generating function for $E_n(x)$:
    \[E(x,t) \coloneqq \sum_{j \ge 0} E_j(x)t^j.\]
    By Lemma~\ref{lem:2.8} and Proposition~\ref{prop:4.1}, for each integer $k \ge 1$, we have
    \[E(k,t) = \sum_{j \ge 0} h_j(1, 2, \ldots, k)t^j = \prod_{i=1}^k \frac{1}{1-it}.\] 
    We want to compute $[x]E_n(x)=[x][t^n]E(x,t)=[t^n][x]E(x,t)$.
    Since \eqref{eq_ehrmeno} states that $E_n(0)=0$ for all integers $n \ge 1$, we have $E(0,t)= E_0(0)=1$. Therefore, by Remark~\ref{rem:4.3}, it follows that $[x]E(x,t) = [x]\ln(E(x,t))$, yielding
    \[[x]E_n(x)=[t^n][x]\ln(E(x,t)).\]
    We compute $\ln(E(k,t))$ for integers $k \ge 1$:
    \begin{multline*}
    \ln(E(k,t)) = \\\ln\left(\prod_{i=1}^k\frac{1}{1-it}\right)= 
    -\sum_{i=1}^k \ln (1-it)=
    \sum_{i=1}^k \sum_{j \ge 1} \frac{(it)^j}{j}=
    \sum_{j \ge 1} \frac{t^j}{j} \sum_{i=1}^k i^j =
     \sum_{j \ge 1} \frac{t^j}{j} F_j(k).
\end{multline*}
    Consequently, we have shown that $[t^n] \ln(E(k, t)) = \frac{F_n(k)}{n}$ for $n \ge 1$ and for infinitely many integer values of $k$. Since $E(x,0)=E_0(x) \equiv 1$, the coefficient $[t^n] \ln(E(x, t))$ is a finite combination of the polynomials $\{E_i(x)\}_{i \ge 1}$ for each $n \ge 1$, and is therefore itself a polynomial in $x$. Since both $[t^n] \ln(E(x, t))$ and $\frac{F_n(x)}{n}$ are polynomials that coincide on an infinite set of points, the Polynomial Identity Theorem guarantees that the relation holds for all real values of $x$. Therefore, \[ [x][t^n]\ln(E(x, t)) = [x] \frac{F_n(x)}{n} = \frac{B_n}{n},\]
    where the last equality follows from \eqref{Bernoulli}. This completes the proof.
\end{proof}

\begin{proposition}\label{prop_coeff_2}
    For each integer $n \ge 1$, the following identity holds
   \[ [x]\Ehr_{P_n}(x)=n+\sum_{j=1}^n\frac{B_j}{j}.\]
\end{proposition}

\begin{proof}
    Setting $E_n(x) \coloneqq \Ehr_{P_n}(x)$, the proof proceeds as in the previous case. The only difference is that $E_n(0)=1$, which implies \[E(0,t) = \sum_{j \ge 0}t^j=\frac{1}{1-t}.\]
    Therefore, before applying the logarithmic series expansion, we must factor out $E(0,t)$:
    \begin{align*}
     &\ln(E(x,t))=\ln\left(\frac{1}{1-t}\right)+\ln((1-t)E(x,t)),\\
    &a(t) \coloneqq [x]\ln(E(x,t))= 0+[x]\ln((1-t)E(x,t))=[x](1-t)E(x,t),\\
    &[x]E_n(x)=[t^n][x]E(x,t)=[t^n]\frac{1}{1-t}a(t)=\sum_{j=0}^n[t^j]a(t)=\sum_{j=1}^n[t^j]a(t).   
    \end{align*}
     
     \noindent Proceeding as before, we find
     \[E(k,t) = \prod_{i=1}^{k+1} \frac{1}{1-it},\]
     \[ [t^n]a(t) = [x] \frac{F_n(x+1)}{n} = [x]\frac{F_n(x)+(x+1)^n}{n}= \frac{B_n+n}{n} \quad \text{for }n \ge 1.\]

     \noindent It follows that
     \[[x]E_n(x)=\sum_{j = 1}^n[t^j]a(t)=\sum_{j = 1}^n\frac{B_j+j}{j}=n+\sum_{j = 1}^n\frac{B_j}{j}.\qedhere\]
\end{proof}

\subsection{\texorpdfstring{The {\boldmath $h^*$}-polynomial}{The h*-polynomial}}\label{sec:4.2} In their paper,
Gessel and Stanley \cite{GesselStanley} also investigated the properties of the so-called \textbf{Stirling polynomials}, which are defined for each non-negative integer $n$ by
\[f_n(k) \coloneqq \left\{ \begin{matrix} k+n \\ k \end{matrix} \right\}.\]
They showed that $f_n(k)$ is indeed a polynomial in $k$; consequently, by the Polynomial Identity Theorem, the equality in Corollary~\ref{cor:4.2} holds for all $n \in \Z_{\ge 0}$ and $k \in \mathbb{C}$. 
Moreover, they established the following result, which can be reinterpreted as the calculation of the $h^*$-polynomial of the comb poset:

\begin{proposition}\label{prop_h}
{\cite[Theorem 2.1]{GesselStanley}}
For every integer $n \ge 1$, the generating function for the Stirling polynomials $f_n(x)$ is given by
\[\sum_{j \ge 0} f_n(j) t^{j} = \frac{\sum_{j=1}^{n} \myeuler{n}{j-1} t^j}{(1-t)^{2n+1}}.  
\]
\end{proposition}

\noindent Thanks to Corollary~\ref{cor:4.2} and the condition \eqref{eq_ehrmeno} (or $f_n(0)=0)$, we have
\[ 
\sum_{j \ge 0} \Ehr_{P_n}(j) t^{j} = \frac{\sum_{j=0}^{n-1} \myeuler{n}{j} t^{j}}{(1-t)^{2n+1}}. 
\]
That is, for $n \ge 1$, the $h^*$-polynomial of the comb polytope $P_n$ is given by $h^*_j = \myeuler{n}{j}$. In particular, $h_j^*=0$ for $ n \le j \le 2n$.

\subsection{Conclusion of the proof}
\begin{proof}
For $n \ge 1$, Propositions~\ref{Prop_coeff} and~\ref{prop_h} yield  $[x]\Ehr_{P_n}(x-1) = \textstyle \frac{B_n}{n}$ and $h^*_j = \myeuler{n}{j}$, respectively. Thus, by Lemma~\ref{Lem_F} (1), we obtain the first identity of the theorem:
\[\frac{B_n}{n} = \frac{1}{2n} \sum_{j=0}^{n-1} (-1)^j \frac{\myeuler{n}{j}}{\binom{2n-1}{j}}.\]

\noindent  For the second identity, Proposition~\ref{prop_coeff_2} yields $[x]\Ehr_{P_n}(x)=n+\sum_{j=1}^n\frac{B_j}{j}$. Thus, by Lemma~\ref{Lem_F} (2), we have
\[n+\sum_{j=1}^n\frac{B_j}{j}=H_{2n} + \frac{1}{2n} \sum_{j=1}^{n-1} (-1)^{j-1} \frac{\myeuler{n}{j}}{\binom{2n-1}{j-1}}.\]
\end{proof}

\begin{remark}
    An analogous proof holds if we consider the open comb polytope $\mathring{P}_n$ instead of $P_n$. In particular, note that \begin{equation}
        \Ehr_{\mathring{P}_n}(x+1) = \left[ \genfrac{}{}{0pt}{}{x}{x-n} \right],
    \end{equation} where the square brackets denote the (unsigned) Stirling numbers of the first kind (see \cite[Section 6.1]{GKP94}).
\end{remark}

\section{Towards generalizations}

\noindent In this section, we discuss how the method used to study the Ehrhart polynomial of the comb polytope is more general than it might seem. We now define a class of posets for which the same approach is applicable.

\begin{definition}
    
Let $P$ be a finite poset. We denote by $\hat{P}$ the poset obtained by adjoining a maximum element $\hat{1}$ to $P$. 
For a non-negative integer $n$, the \textbf{comb over $P$}, denoted by $\comb_n(P)$, is the poset consisting of $n$ disjoint copies of $\hat{P}$, where the partial order is determined by the relations in each $\hat{P}_i$ together with the chain of maximum elements:
$\hat{1}_1 < \hat{1}_2 < \cdots < \hat{1}_n.$
\end{definition}

\noindent We now give a more general definition.

\begin{definition}
     Given two finite posets $P$ and $R$, let $S$ be the poset obtained by adjoining an element $c$ which is a maximum for $P$ and a minimum for $R$. We define the \textbf{bicomb over $P$ and $R$} as the poset $\comb_n(P,R)$ consisting of $n$ disjoint copies of $S$ in which the connecting elements $\{c_1, \ldots, c_n\}$ form a chain: $c_1 < c_2 < \cdots < c_n$.
\end{definition}

\begin{remark}
Observe that $\comb_n(P, \varnothing) = \comb_n(P)$, recovering the standard comb. Note also that $\comb_1(P)$ simply adds a maximum element to $P$; as an order polytope, this is the pyramid over $\O(P)$ (see Section~\ref{sec_pyr}).
\end{remark}

\noindent We call the posets arising from this construction \textbf{comb posets}, and their associated order polytopes \textbf{comb polytopes}. See Figure~\ref{poset_1} and the following figures for some simple examples.

For comb posets, we can compute the Ehrhart polynomials by following the strategy employed in Proposition~\ref{prop:4.1}, simply by adjusting the specialization of the polynomial $h_n$.
Indeed, by summing over all valid integer sequences $0 \leq c_1 \leq \cdots \leq c_n \leq k$ for the connecting chain, weighted by the number of possible choices for the remaining coordinates, we obtain the following result.

\begin{proposition}
\phantomsection
\label{b}
Let $n$ and $k$ be non-negative integers, let $P$ and $R$ be finite posets, and let $Q_n \coloneqq \comb_n(P,R)$ be the bicomb over $P$ and $R$. The Ehrhart polynomial $\Ehr_{\O(Q_n)}$ evaluated at $k$ coincides with the polynomial $h_n(x_0, \ldots, x_k)$ under the specialization
\[x_i = \Ehr_{\O(P)}(i) \cdot \Ehr_{\O(R)}(k-i) \]
for each $i \in \{0, \ldots, k\}$.
\end{proposition}

\begin{corollary}\label{cor_ehrcomb}
Let $P$ be a finite poset. For all non-negative integers $n$ and $k$, we have
   \[\Ehr_{\O(\comb_n(P))}(k)=h_n(\Ehr_{\O(P)}(0),\ldots,\Ehr_{\O(P)}(k)).\]
\end{corollary}

\begin{remark}
If we want to study open comb polytopes, an analogous statement holds involving the elementary symmetric polynomial $e_n$ instead of $h_n$.
\end{remark}

\begin{example}\label{es}
    Let $P = \{x, y\}$ be an antichain (i.e., a poset without relations) and let $Q_n\coloneqq\comb_n(P)$ be the comb over $P$ (see Figure~\ref{poset_3}). The Ehrhart polynomial of $\O(P)$ is $\Ehr_{\O(P)}(x) = (x+1)^2$; therefore $\Ehr_{\O(Q_n)}(k) = h_n(1,4,\ldots,(k+1)^2)$
    for every non-negative integer $k$.
\end{example}

\noindent Furthermore, the strategy used in Proposition~\ref{Prop_coeff} can be applied to extract the linear coefficient of the Ehrhart polynomial of a comb polytope.

\begin{proposition}
\label{prop_coeff_gen}
Let $P$ be a finite poset, $n \ge 1$ an integer, and $Q_n = \comb_n(P)$. For each $i \ge 1$, let $S_i(x)$ be the unique polynomial such that
\[ S_i(k) = \sum_{j=0}^{k} (\Ehr_{\O(P)}(j))^i \]
for all non-negative integers $k$. Then
    \[[x] \Ehr_{\O(Q_n)}(x-1) = \frac{1}{n} [x] S_n(x-1),\]
   \[[x] \Ehr_{\O(Q_n)}(x) = \sum_{i=1}^n  \frac{[x]S_i(x)}{i}.\]
\end{proposition}

\begin{proof}
To simplify the notation, we set $q(x)=\Ehr_{\O(P)}(x)$. Thanks to Corollary \ref{cor_ehrcomb}, the proof is identical to those of Propositions~\ref{Prop_coeff} and~\ref{prop_coeff_2}, simply by replacing the specialization\footnote{
Note that for the comb poset $P_n$ we had $q(i)=
i+1$ and $S_n(x) = F_n(x+1)$.
}
$i+1$ with $q(i)$: \[E(k,t) \coloneqq \sum_{j \ge 0} \Ehr_{\O(Q_j)}(k-1)t^j = \sum_{j \ge 0} h_j(q(0),\ldots,q(k-1))t^j = \prod_{i=0}^{k-1} \frac{1}{1-q(i)t}.\]
\end{proof}

\begin{example}
    Let $Q_n$ be the poset defined in Example~\ref{es}, for which it holds $S_i(k) = \sum_{j=0}^{k} (j+1)^{2i} = F_{2i}(k+1)$ for every non-negative integer $k$. Then, we have
\[[x]\Ehr_{\O(Q_n)}(x-1)=\frac{1}{n}[x]F_{2n}(x)=\frac{B_{2n}}{n},\]
\[[x] \Ehr_{\O(Q_n)}(x) = \sum_{i=1}^n \frac{[x]F_{2i}(x+1)}{i}=2n+\sum_{i=1}^n \frac{B_{2i}}{i}.\]
\end{example}

The study of various posets through this approach is of particular interest because they are related to specific classes of permutations (their linear extensions), and via these propositions we can obtain identities involving their descent enumerators (the $h^*$-polynomial of their order polytopes, as noted in Remark~\ref{rem_des}).
This is exactly what we did in the previous section with the comb poset $C_n$ to prove Theorem~\ref{thm:1.1}.

As an example, let us consider the bicomb over two chains (i.e., totally ordered sets; see Figure~\ref{poset_1}), whose linear extensions are permutations ordered by the $j$-th occurrence.

\begin{definition}
    A permutation of the multiset $\{1^r, 2^r, \ldots, n^r\}$ is said to be \textbf{ordered by the $j$-th occurrence} if, in its one-line notation, the $j$-th occurrence of $i$ appears before the $j$-th occurrence of $i+1$, for all $1 \le i < n$.
    We denote the set of all such permutations by $\mathcal{S}_{n,j}^{(r)}$.
\end{definition}

\noindent For example, $ 13\underline{1}2\underline{23}312 \in \mathcal{S}_{3,2}^{(3)}$ is ordered by the second occurrence. 

In particular, the linear extensions of the bicomb $\comb_n(\Ch_{j-1}, \Ch_{r-j})$ over two chains of $j-1$ and $r-j$ elements are in a descent-preserving bijection with $\mathcal{S}^{(r)}_{n,j}$.

\begin{remark}
   The comb poset $C_n$ discussed in the previous section is the comb over the singleton poset $\{x\}$, which is simply a chain of one element $\Ch_1 = \{x\}$. Thus, $C_n = \comb_n(\Ch_1)$.
   This connection is clarified by Park \cite[Proposition 2.18]{Park}, who provides an explicit descent-preserving bijection between Stirling permutations $\mathcal{S}^{(2)}_{n}$ and the set of permutations ordered by the second occurrence $\mathcal{S}^{(2)}_{n,2}$; this gives an alternative proof of Proposition~\ref{prop_h} (by Remark~\ref{rem_des}).
\end{remark}

\noindent Moreover, if we consider the comb $\comb_n(\Ch_{k-1})$ over a single chain of $k-1$ elements (see Figure~\ref{poset_2}), its linear extensions correspond to $k$-uniform Restricted Growth Functions of length $kn$.

\begin{definition}[Milne~\cite{RGF}]
A \textbf{Restricted Growth Function} (RGF) of length $n$ is a sequence of positive integers $w = (w_1, w_2, \ldots, w_n)$ such that:
\begin{enumerate}[($i$)]
    \item $w_1 = 1$;
    \item $w_{i+1} \le \max\{w_1, \ldots, w_i\} + 1$ for each $1 \le i < n$.
\end{enumerate}
 The set of RGFs of length $n$ is in bijection with the set of partitions of $\{1, \ldots, n\}$: given a partition $\Pi = \{B_1, \ldots, B_r\}$ with blocks ordered by their minimum elements, the corresponding RGF $w$ is defined by $w_i = j$ if $i \in B_j$.
\end{definition}

\begin{definition}
    An RGF $w$ is \textbf{\boldmath $k$-uniform} if each value appearing in the sequence $w$ appears exactly $k$ times. In terms of the corresponding partition $\Pi$, this is equivalent to requiring that every block $B_j$ has cardinality $k$.
\end{definition}

\noindent Specifically, for each $k$-uniform RGF $w$ of length $kn$, the map $w \mapsto w'$ defined by $w'_j = n + 1 - w_{kn+1-j}$ is a descent-preserving bijection onto $\mathcal{S}_{n,k}^{(k)}$. As an example, for $n=3$ and $k=2$, the RGF $w=121332$ (associated with the partition $\{\{1,3\},\{2,6\},\{4,5\}\}$) is mapped to $w'=21\underline{1}3\underline{23}$.

\begin{figure}[b] 
    \centering

    \begin{minipage}{\textwidth}
        \centering
        \begin{tikzpicture}[
            every node/.style={circle, fill=black, inner sep=1.5pt},
            scale=0.6
        ]
            \def\dx{1.6} \def\dy{1.1} \def\h{0.9}
            \foreach \i in {0, ..., 4} {
                \node (X\i) at (\i*\dx, \i*\dy) {};
                \node (T\i) at (\i*\dx, \i*\dy + \h) {};
                \node (Y\i) at (\i*\dx, \i*\dy - \h) {};
                \node (Z\i) at (\i*\dx, \i*\dy - 2*\h) {};
                
                \draw (Z\i) -- (Y\i) -- (X\i) -- (T\i);
                \ifnum \i > 0
                    \pgfmathtruncatemacro{\prev}{\i-1}
                    \draw (X\prev) -- (X\i);
                \fi
            }
        \end{tikzpicture}
        \caption{Bicomb over two chains: $\{x < y\}$ and $\{z\}$.}
        \label{poset_1}
    \end{minipage}

    \vspace{1cm}

    \begin{minipage}{0.48\textwidth}
        \centering
        \begin{tikzpicture}[
            every node/.style={circle, fill=black, inner sep=1.3pt},
            scale=0.5
        ]
            \def\dx{1.6} \def\dy{1.1} \def\h{0.9}
            \foreach \i in {0, ..., 4} {
                \node (X\i) at (\i*\dx, \i*\dy) {};
                \node (Y\i) at (\i*\dx, \i*\dy - \h) {};
                \node (Z\i) at (\i*\dx, \i*\dy - 2*\h) {};
                
                \draw (Z\i) -- (Y\i) -- (X\i);
                \ifnum \i > 0
                    \pgfmathtruncatemacro{\prev}{\i-1}
                    \draw (X\prev) -- (X\i);
                \fi
            }
        \end{tikzpicture}
        \caption{Comb over the chain $\{x < y\}$.} 
        \label{poset_2}
    \end{minipage}
    \hfill
    \begin{minipage}{0.48\textwidth}
        \centering
        \begin{tikzpicture}[
            every node/.style={circle, fill=black, inner sep=1.3pt},
            scale=0.5
        ]
            \def\dx{2.0} \def\dy{1.1} \def\h{0.9} \def\off{0.5}
            \foreach \i in {0, ..., 4} {
                \node (X\i) at (\i*\dx, \i*\dy) {};
                \node (Y\i) at (\i*\dx - \off, \i*\dy - \h) {};
                \node (Z\i) at (\i*\dx + \off, \i*\dy - \h) {};
                
                \draw (Y\i) -- (X\i);
                \draw (Z\i) -- (X\i);
                \ifnum \i > 0
                    \pgfmathtruncatemacro{\prev}{\i-1}
                    \draw (X\prev) -- (X\i);
                \fi
            }
        \end{tikzpicture}
        \caption{Comb over the antichain $\{x, y\}$.}
        \label{poset_3}
    \end{minipage}
\end{figure}

\begin{example}   
To compute formulas for the descent enumerator of $k$-uniform RGFs of length $kn$, we need to study the related poset $\comb_n(\Ch_{k-1})$, where $\Ch_{k-1}$ denotes a chain of $k-1$ elements.
In particular, we compute the linear coefficient of its associated Ehrhart polynomial and apply Lemma~\ref{Lem_F}.

Since $\Ehr_{\O(\Ch_{k-1})}(x) = \binom{x+k-1}{k-1}$, for every non-negative integer $t$, we have\[
S_n(t) = \sum_{j=0}^t \binom{j+k-1}{k-1}^n,\]
where $S_n$ is the polynomial defined in Proposition~\ref{prop_coeff_gen}.
By expanding the summand and expressing $S_n(t)$ as a linear combination of Faulhaber polynomials, we can compute the linear coefficient $[x]S_n(x-1)$, from which the desired linear coefficient also follows.
\end{example}

\section{Open problems}

\subsection{Lattice pyramids}\label{sec_pyr}
Given a lattice polytope $\P \subset \R^n$, the
\textbf{pyramid over} $\P$, denoted by $\Pyr(\P)$, is the polytope defined by
\[\Pyr(\P) = \operatorname{conv}\left( (\P \times \{0\}) \cup \{v\} \right) \subset \R^{n+1},\]
where $v = (0, \ldots, 0, 1)$ \cite[Section 2.4]{ccd}. It is well known that the $h^*$\nobreakdash-polynomial is invariant under this construction, whereas the Ehrhart polynomial satisfies the following summation property for each non-negative integer $k$:
\[\Ehr_{\Pyr(\P)}(k) = \sum_{j=0}^k \Ehr_\P(j).\]

In the specific case of an order polytope $\O(P)$, the pyramid $\Pyr(\O(P))$ is equivalent to the order polytope $\O(\hat{P})$, where $\hat{P}$ is the poset obtained by adjoining a maximum element to $P$.

In \cite{Ferroni2026}, the study of the pyramid over $\P$ is employed to derive new combinatorial identities.
Following this line of research, a natural question is to study the Ehrhart polynomial of the pyramid over the comb polytope $P_n$ (independently of the $h^*$-polynomial, which we have already determined) and to compute
\[[x]\Ehr_{\Pyr(P_n)}(x-1),\]
yielding new identities involving the second Eulerian numbers.

\subsection{Higher-order Stirling permutations}

In his paper, Park \cite{Park} introduced higher-order Stirling permutations. For a given order $r \ge 1$, we denote the set of all such permutations by $\mathcal{S}^{(r)}_n$. These are defined as permutations of the multiset $\{1^r, 2^r, \ldots, n^r\}$ such that, for each $i \in \{1, \ldots, n\}$, all elements that appear between every pair of occurrences of $i$ are at least $i$.

Moreover, Park studied the Stirling poset $P^{(2)}_k$ (which we call the comb poset $C_{k}$) and its generalization $P^{(r)}_k$ (see Figure~\ref{fig_2}). He showed that the $h^*$-polynomial of the order polytope $\O(P^{(r)}_k)$ is the descent enumerator of Stirling permutations of order $r$.
A natural direction for future research would be the study of $[x]\Ehr_{\O(P^{(r)}_k)}(x-1)$ for $r \ge 3$, with the aim of finding identities involving higher-order Eulerian numbers.

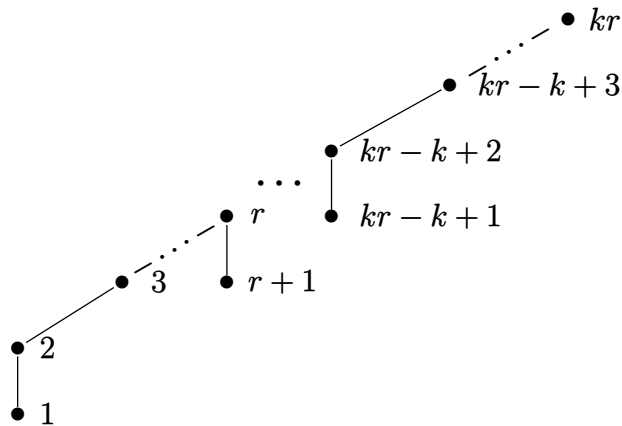
\begin{figure}[hb]
    \centering
\[\begin{tikzcd}[column sep=small, row sep=small]
	&&&&&&&&&& |[inner sep=0pt]| \bullet && \phantom{\bullet} \\
	&&&&&&&& |[inner sep=0pt]| \bullet  && \phantom{\bullet}  \\
	&&&&&& |[inner sep=0pt]|\bullet && \phantom{\bullet} \\
	&&&& |[inner sep=0pt]|\bullet  && |[inner sep=0pt]|\bullet && \phantom{\bullet}   \\
	&& |[inner sep=0pt]|\bullet && |[inner sep=0pt]|\bullet && \phantom{\bullet}  \\
	|[inner sep=0pt]|\bullet  && \phantom{\bullet} \\
    |[inner sep=0pt]|\bullet && \phantom{\bullet} 
	\arrow[phantom, "- \dots -" rotate=30, from=2-9, to=1-11]
	\arrow[no head, from=3-7, to=2-9]
	\arrow[no head, from=3-7, to=4-7]
	\arrow[phantom, "- \dots -" rotate=30, from=5-3, to=4-5]
	\arrow[no head, from=6-1, to=5-3]
	\arrow[no head, from=6-1, to=7-1]
    \arrow[no head, from=4-5, to=5-5]
    \arrow[phantom, "{\boldsymbol{\cdot\cdot\cdot}}", from=4-5, to=3-7]
    \arrow[phantom, "{1 \ \ \ }", from=7-1, to=7-3]
    \arrow[phantom, "{2 \ \ \ }", from=6-1, to=6-3]
    \arrow[phantom, "{3 \ \ \ }", from=5-3, to=5-5]
    \arrow[phantom, "{r \ \ \ \ }", from=4-5, to=4-7]
    \arrow[phantom, "{\ \ r+1}", from=5-5, to=5-7]
    \arrow[phantom, "{ \ \ \ \ \ \ \ \ \ kr-k+1}", from=4-7, to=4-9]
    \arrow[phantom, "{ \ \ \ \ \ \ \ \ \ kr-k+2}", from=3-7, to=3-9]
    \arrow[phantom, "{\ \ \ \ \ \ \ \ \  kr-k+3}", from=2-9, to=2-11]
    \arrow[phantom, "{kr \ \ \ }", from=1-11, to=1-13]
\end{tikzcd}\]
\caption{The Stirling poset $P^{(r)}_k$}
 \label{fig_2}
\end{figure}

\begin{remark}
    In \cite{Ferroni2026}, the author studies the case of the hypercube, which is the order polytope associated with the antichain; the latter can be seen as the base case $P^{(1)}_k$.
\end{remark}

\bibliographystyle{amsalpha}
\bibliography{bibliography}

\end{document}